\documentclass[12pt]{elsarticle}

\usepackage{graphicx}
\usepackage{amssymb}
\usepackage{amsmath}
\usepackage{amsthm}

\usepackage{csquotes}
\usepackage{algorithm}
\usepackage{algorithmic}
\usepackage{mathtools}

\makeatletter
\def\BState{\State\hskip-\ALG@thistlm}
\makeatother

\usepackage{mathtools}

\newcommand{\bZer}{\boldsymbol{0}}
\newcommand{\bOne}{\boldsymbol{1}}

\newcommand{\bu}{\boldsymbol{u}}
\newcommand{\bv}{\boldsymbol{v}}
\newcommand{\bp}{\boldsymbol{p}}
\newcommand{\bq}{\boldsymbol{q}}

\newcommand{\Zl}{\mathbb{Z}_{\ell}}
\newcommand{\Zlt}{\mathbb{Z}_{\ell}^{\times}}

\newcommand{\bA}{\boldsymbol{A}}

\newcommand{\bH}{\boldsymbol{H}}
\newcommand{\bI}{\boldsymbol{I}}

\newcommand{\PSD}{\text{PSD}}

\newcommand{\bmu}{\boldsymbol{\mu}}

\newcommand{\nchoosek}[2]{\left(\begin{matrix}
#1 \\ #2
\end{matrix} \right)}

\newtheorem{thm}{Theorem}[section]

\newtheorem{crly}[thm]{Corollary}
\newtheorem{dfn}[thm]{Definition}

\newtheorem*{thm*}{Theorem}

\begin{document}

\begin{frontmatter}


\title{Legendre pair of length $77$ using complementary binary matrices with fixed marginals}

\author[AFIT]{Jonathan~S.~Turner}
\ead{jonathan.turner@afit.edu}
\author[WLU]{Ilias~S.~Kotsireas}
\ead{ikotsire@wlu.ca}
\author[AFIT]{Dursun~A.~Bulutoglu}
\ead{dursun.bulutoglu@afit.edu}
\author[AFIT]{Andrew~J.~Geyer}
\ead{andrew.geyer@afit.edu}

\address[AFIT]{Department of Mathematics and Statistics, Air Force Institute of Technology,\\Wright-Patterson Air Force Base, Ohio 45433, USA}
\address[WLU]{Department of Physics
and Computer Science, Wilfrid Laurier University, Waterloo, Ontario N2L 3C5, CANADA}
\cortext[cor1]{Corresponding author}
\journal{Designs Codes and Cryptography}


\begin{abstract}
We provide a search method for Legendre pairs of composite length based on generating binary matrices with fixed row and column sums from compressed, complementary integer vectors.  This approach yielded the first construction of a Legendre pair of length $77$, as well as the first exhaustive generation of Legendre pairs of length $55$.  
\end{abstract}


\end{frontmatter}

\section{Background}
\label{Sec:Background}

Two vectors, $\bu$ and $\bv$, indexed by $\Zl$  are called {\em complementary sequences} if and only if there exists some $\lambda$ such that :
\begin{equation}\label{eqn:Complementary}
    \sum_{j=0}^{\ell-1}v_jv_{j+g}+u_ju_{j+g}=\lambda \quad  \forall  g\neq 0, \ g\in\Zl,
\end{equation}
see \cite{Fletcher2001}.

The {\em density} of a \{0,1\} vector, $\bv$, of length $\ell$ is $\kappa=\sum_{j=0}^{\ell}v_j$. Two $\{0,1\}$ complementary sequences of length $\ell$ constitute a $\{0,1\}-${\em Legendre pair} (LP) if $\kappa=\lambda=(\ell+1)/2$.  Common transformations of LP are listed for completeness along with the corresponding constants.  $\{0,1\}-$LP with $\kappa=(\ell-1)/2$ have $\lambda=(\ell-3)/2$.  Transforming to $\{-1,1\}-$LP require $\kappa=\pm 1$ with  $\lambda=-2$.

A \textit{Hadamard matrix} $\bH$ is a $n \times n$ matrix of $\pm 1$'s such that $\bH \bH^{\top}=n\bI$. It easy to show that $n$ must be divisible by $4$ for such a matrix to exist.  There exists a construction for Hadamard matrices with size $n=2\ell+2$ based on LP.  Just as Hadamard matrices are conjectured to exist for all $n$ divisible by $4$, LP are conjectured to exist for all odd $\ell$ \cite{balonin2020three}.  The smallest unresolved case for LP has long been $\ell=77$ \cite{Fletcher2001}.  Furthermore, exhaustive searches have only been conducted for $\ell\leq 45$ \cite{Fletcher2001}.  

A {\em circulant shift} of a vector $\bv$ by $j\in \Zl$, denoted by $c_j(\bv)$, is the permutation of indices such that $\left(c_j(\bv)\right)_{g}=v_{g-j}$.  A {\em necklace} is an equivalence class of vectors of length $\ell$ under circulant shifts, $j\in\Zl$ \cite{Dokovic2015, Sawada2013}.  Let $\Zlt =\{j\in\Zl\,\mid  \,\gcd(j,\ell)=1\}$ be the multiplicative group of $\Zl$.  A {\em decimation} of a vector $\bv$ by $k\in \Zlt $, denoted by $d_k(\bv)$, is the permutation of indices such that $\left(d_k(\bv)\right)_{g}=v_{k^{-1}g}$ where $k^{-1}$ denotes the multiplicative inverse of $k$ in $\Zlt $.  A {\em decimation class} is an equivalence class of vectors under circulant shifts and decimations \cite{Fletcher2001}.  For each $\bu'$ in the decimation class of $\bu$, there exists some $\bv'$ in the decimation class of $\bv$ that is complementary to $\bu'$ if and only if $\bu$ and $\bv$ are complementary sequences \cite{Fletcher2001}.  We call $(\bv,\bu)$ and $(\bv',\bu')$ {\em equivalent} complementary sequences.  

Computational searches for LP focus on discovering new necessary constraints for generating vectors satisfying equation~\eqref{eqn:Complementary}, or similarly applying constraints that only remove equivalent LP.  A constraint called the {\em PSD-test} significantly reduces the search space \cite{Fletcher2001}.  The discrete Fourier transform (DFT) of $\bv$ is $\bmu$, where $\mu_k=\sum_{j=0}^{\ell-1}v_j\omega^{jk}$ for ${k=\{0,1,\dots,\ell-1\}}$ and $\omega_{\ell}=e^{2\pi i/\ell}$.  The {\em power spectral density} (PSD) of a vector $\bv$ with DFT $\bmu$ is $\text{PSD}(\bv)$ where 
\begin{equation*}
    \text{PSD}(\bv,j)=\mu_j\bar{\mu}_j \quad  \text{for each } \ j\in\Zl.
\end{equation*}
Let $\bu$ have DFT $\boldsymbol{\nu}$.  By the Weiner-Khinchin theorem, equation~\eqref{eqn:Complementary} transforms to
\begin{equation}\label{eqn:CompPSD}
    \sum_{j=0}^{\ell-1}\mu_j\bar{\mu}_{j+g}+\nu_j\bar{\nu}_{j+g}=PSD(\bv,g)+PSD(\bu,g)=\gamma \quad  \forall  g\neq 0, \ g\in\Zl,
\end{equation}
where $\gamma=(\ell+1)/2$ $(\gamma=2\ell+2)$ for $\{0,1\}-$LP ($\{-1,1\}-$ LP) \cite{Fletcher2001}.  The PSD-test follows directly from equation~\eqref{eqn:CompPSD} and states if $\bv$ forms an LP with some $\bu$, then $PSD(\bv,g)<\gamma$ for all $g\in\Zl/\{0\}$ \cite{Fletcher2001}.  

Recently, a set of constraints was developed that enforces equation \eqref{eqn:Complementary} on indices $g\not\in\Zl^\times$ based on $\delta-$modular compression.  For a vector, $\bv$, of length, $\ell=\delta_1\delta_2$, the {\it $\delta_1$-compression} of $\bv$ is vector $\bq$ of length $\delta_1$ such that $q_g=\sum_{j=0}^{\delta_2-1}v_{g+j\delta_1}$ for $g\in\mathbb{Z}_{\delta_1}$ \cite{Dokovic2015compression}.  We also say $\bv$ $\delta_1$-compresses to $\bq$.  It follows if $\bu$ and $\bv$ are complementary sequences with $\delta_1$-compressions $\bp$ and $\bq$, then $\bp$ and $\bq$ are also complementary sequences \cite{Dokovic2015compression}. That is, they satisfy equations~\eqref{eqn:Complementary} and \eqref{eqn:CompPSD}.  Moreover, 
\begin{equation} \label{eqn:CompConstPAF}
     \sum_{j=0}^{\delta_1-1}p_jp_{j+g}+q_jq_{j+g}=\delta_2\lambda \quad  \forall  g\neq 0, \ g\in\mathbb{Z}_{\delta_1} 
\end{equation}
and
\begin{equation} \label{eqn:CompConstPAF0}     
     \sum_{j=0}^{\delta_1-1}p_j^2+q_j^2=2\delta_1\sum_{k=0}^{\ell-1}v_i^2=2\delta_1\kappa \quad  \forall  g\neq 0,
\end{equation}
or equivalently,
\begin{equation}\label{eqn:CompConstPSD}
     PSD(\bp,g)+PSD(\bq,g)=\gamma\quad  \forall  g\neq 0, \ g\in\mathbb{Z}_{\delta_1},
\end{equation}
see \cite{Dokovic2015compression}.  

This constraint set essentially divides the problem into two stages.  The first is the generation of all complementary, integer vectors of length $\delta_1$ with elements in $\{0,1,\dots,(\ell/\delta_1)\}$.  The second stage is the {\em decompression} of these compressed pairs, followed by a search for binary complementary sequences.  This relation is used to great effect in the search for complementary sequences known as periodic Golay pairs \cite{Dokovic2015}.

We expand the concept of $\delta$-modular compression to increase the set of constraints when $\ell=\delta_1\delta_2$ with $\gcd(\delta_1,\delta_2)=1$.  This is the case for the remaining unsolved cases for LP existence with $\ell<200$, excluding $169=13^2$.  The cases for
{\small{\begin{equation*}
    \ell\in\{77,85,87,115,117,129,133,145,147,159,161,169,175,177,185,187,195\}
\end{equation*}}}
constitute the current list of unsolved cases \cite{balonin2020three}.  The result of this expansion are pairs of binary matrices with fixed row and column sums, also known as binary matrices with fixed marginals \cite{brualdi1980matrices}.  We prove various features on the spaces of these matrices with respect to decimation classes to further reduce the space of compressed complementary sequences.  We conclude with the first exhaustive generation of LP of size  $\ell=55$ and the first discovery of a size $\ell=77$ LP.
 
 \section{$\delta$-Modular compression}

The benefits of compression are most emphasized by PSD.  Notice equation~\eqref{eqn:CompConstPSD} does not alter $\gamma$, and does not require any additional constraint on the $0$th index, as in equation~\eqref{eqn:CompConstPAF0}.  This is because PSD values are {\em preserved} by $\delta$-modular compression.  The PSD vector of $\bv$ contains the PSD vector of its $\delta$-modular compression~\cite{Dokovic2015compression}. To prove this and facilitate understanding, it is expedient to first prove a stronger result regarding the DFT.    

\begin{thm}
\label{thm:CompPresMu}
Let $\bv$ be a vector of length $\ell=\delta_1\delta_2$, and $\bq$ be the $\delta_1$-compression of $\bv$.  Let $\bmu$ be the DFT of $\bv$, and $\boldsymbol{\nu}$ be the DFT of $\bq$.  Then $\nu_g=\mu_{g\delta_2}$
\end{thm}

\begin{proof}
  By definition, $\mu_{g\delta_2}=\sum_{j=0}^{\ell-1}v_j\omega_{\ell}^{g\delta_2j}$.  Notice ${\omega_{\ell}^{\delta_2}=e^{2\pi i \delta_2/\ell}=e^{2\pi i / \delta_1}=\omega_{\delta_1}}$.  Then $\mu_{g\delta_2}=\sum_{j=0}^{\ell-1}v_j\omega_{\delta_1}^{gj}$.  By definition, $\omega_{\delta_1}^{j}=\omega_{\delta_1}^{(j+\delta_1)}$  and $$\mu_{g\delta_2}=\sum_{j=0}^{\delta_1-1}\omega_{\delta_1}^{gj}\left(\sum_{k=0}^{\delta_2-1}v_{k\delta_1+j}\right).$$  Since $\bq$ is the $\delta_1$-compression of $\bv$, then $q_j=\sum_{k=0}^{\delta_2-1}v_{k\delta_1+j}$.  
  Therefore, ${\mu_{g\delta_2}=\sum_{j=0}^{\delta_1-1}\omega_{\delta_1}^{gj}q_j=\nu_g}$.
\end{proof}

\begin{crly}\label{crly:preservePSD}
Let vectors $\bu$ and $\bv$ of length $\ell=\delta_1\delta_2$ be complementary sequences with PSD constant $\gamma$. If $\bp$ and $\bq$ are the respective $\delta_1$-compressions of $\bu$ and $\bv$, then $\bp$ and $\bq$ are complementary sequences with PSD constant $\gamma$. 
\end{crly}  

Corollary~\ref{crly:preservePSD} follows directly from Theorem~\ref{thm:CompPresMu}.  The benefit of this approach arises from the simpler problem of generating complementary positive integer vectors with appropriate density.  These sequences are decompressed into binary vectors satisfying $\delta_1-1$ of the $\ell-1$ equality constraints given by equation~\eqref{eqn:CompPSD}, as well as the required density.  It is well known if $\bu$ and $\bv$ are complementary, then $\bu$ and $c_j(\bv)$ are also complementary for any $j$.   The equivalence class of circulant shifts of a vector is known as a {\em necklace}.  Each cyclic shift of $\bv$ is distinct when $\gcd\left(\ell,\kappa \right)=1$, as is the case with LP~\cite{Sawada2013}.   It follows that $\gcd(\delta_1,\kappa)=1$ for any $\delta_1$ dividing $\ell$, and each compressed necklace has exactly $\delta_1$ distinct vectors.  Thus, reducing compressed vectors to necklace representatives will not preclude any LP.  

This is partially true for decimation class representatives as well.  Reducing to decimation class representatives of the compressed vectors will not preclude any decimation class forming an LP with another.  However, the resultant vectors from decompression may require decimations to form the LP.

\begin{thm}\label{thm:DecClassRep}
Let $\bv$ be a vector of length $\ell=\delta_1\delta_2$ such that $\gcd(\delta_1,\delta_2)=1$, and $\bq$ be the $\delta_1$-compression of $\bv$.  For any $(j,h)\in\mathbb{Z}_{\delta_1}\rtimes\mathbb{Z}_{\delta_1}^\times$, there exists $(k,g)\in \Zl\rtimes\Zlt $ such that $c_k(d_g(\bv))$ has $\delta_1$-compression $c_j(d_h(\bq))$.
\end{thm}

Theorem~\ref{thm:DecClassRep} guarantees that if $\bv$ $\delta_1$-compresses to $\bq$, each member in the decimation class of $\bv$ will $\delta_1$-compress to some member in the decimation class of $\bq$.  This implies that a search for combinatorial objects that are equivalent up to decimation classes may be restricted to decimation class representatives of the compressions to reduce the number of duplicate vectors from each decimation class.  The number of duplicates corresponding to decimations relies on the size of each class' {\em multiplier} subgroup, where $g\in\Zlt $ is called a multiplier of $\bv$ if and only if there exists some $j\in\Zl$ such that $c_j(d_g(\bv))=\bv$.


\begin{crly}\label{crly:DecompDeci}
Let $\bv$ be a vector of length $\ell=\delta_1\delta_2$ such that $\gcd(\delta_1,\delta_2)=1$, and $\bq$ be its $\delta_1$-compression.  Then  $c_j(d_k(\bv))$ $\delta_1$-compresses to $\bq$ for some $j\in\Zl$ if and only if $k \ \text{mod} \ {\delta_1}$ is a multiplier of $\bq$.
\end{crly}

Corollary~\ref{crly:DecompDeci} follows directly from the definition of $\delta_1$-compression.  We denote the set of vectors in the decimation class containing $\bv$ with $\delta_1$- compression $\bq$ as $D_{\bv,\bq}$.  The number of vectors within the same decimation class having the same $\delta_1$-compression is derived in Corollary \ref{crly:DecClassCount}.  

\begin{crly}\label{crly:DecClassCount}
Let $\bv$ be a vector with density relatively prime to length $\ell=\delta_1\delta_2$ where $\gcd(\delta_1,\delta_2)=1$,  and $\bq$ be its $\delta_1$-compression.  Let $G_{\bv}\leq\Zlt $ and $H\leq\mathbb{Z}_{\delta_1}^\times$ be the group of multipliers of $\bv$ and $\bq$ respectively.  Then the number of vectors within the decimation class of $\bv$ that $\delta_1$-compress to $\bq$ is 
\begin{equation*}
    |D_{\bv,\bq}|=\delta_2\frac{\phi(\delta_2)|H|}{|G_{\bv}|}.
\end{equation*}
\end{crly}

\begin{proof}
Since $\gcd(\delta_1,\kappa)=1$ there exists $\delta_2$ circulant shifts of $\bv$ that $\delta_1$-compress to $\bq$.  It remains to show the duplicity of necklaces.  It is well known there exists an isomorphism mapping $\Zlt \mapsto\mathbb{Z}_{\delta_1}^{\times}\times\mathbb{Z}_{\delta_2}^{\times}$.  Since $g\in G_{\bv}$ is a multiplier of $\bv$, it is also a multiplier of $\bq$, as is $g+k\delta_1$ for each $k\in\mathbb{Z}_{\delta_2}^\times$.  It follows only $\phi(\delta_2)=|\mathbb{Z}_{\delta_2}^{\times}|$ such $k$ exist,  there exists $\phi(\delta_2)|H|$ decimations in $\Zlt $ which are multipliers of $\bq$.  Of these, $|G_{\bv}|$ are also multipliers of $\bv$.  Therefore, the number of vectors within the decimation class of $\bv$  that $\delta_1$-compression to $\bq$ is $|D_{\bv,\bq}|=\delta_2\frac{\phi(\delta_2)|H|}{|G_{\bv}|}$.
\end{proof}

Let $Q$ be the set of vectors that $\delta_1$-compress to $\bq$.  A simple counting argument proves the size of this set is
\begin{equation*}
    |Q|=\prod_{j=0}^{\delta_1-1}{\delta_2 \choose q_j}.
\end{equation*}
Let $\boldsymbol{d}_j$ denote the lexicographically smallest vector within its respective decimation class with $\delta_1$-compression $\bq$.  By Corollary~\ref{crly:DecClassCount},
\begin{equation*}
    |Q|=\prod_{j=0}^{\delta_1-1}{\delta_2 \choose q_j}=\sum_{\boldsymbol{d}_j\in Q}|D_{\boldsymbol{d}_j,\bq}|=|H|\sum_{\boldsymbol{d}_j\in Q}\frac{\delta_2\phi(\delta_2)}{|G_{\boldsymbol{d}_j}|}=\sum_{\boldsymbol{d}_j\in Q}\frac{|D_{\boldsymbol{d}_j}|}{|D_{\bq}|}.
\end{equation*}

It follows that generating vectors through decompression of integer decimation class representatives causes significant duplicity for the representatives of the corresponding binary decimation classes.  This duplicity can be reduced by implementing simultaneous decompression on relatively prime factors of $\ell$.

\section{Simultaneous decompression}
\label{sec:SComp}


\begin{dfn}
Let $\ell=\prod_{j=1}^{n}\delta_j$ such that $\gcd(\delta_i,\delta_j)=1$ for $i\neq j$.  Let $\bq^j$ be a vector of length $\delta_j$.  Vector $\bv$ of length $\ell$ is defined to be a {\it simultaneous decompression} of $\{\bq^1, \bq^2,\dots,\bq^n\}$ if and only if $q^j_g=\sum_{k=0}^{(\ell/\delta_j)-1}v_{g+k\delta_j}$, for each $j\in\{1,2,\dots,n\}$ and $g\in\{1,2,\dots,\delta_j\}$.
\end{dfn}

For example, let $\bq^1=[4, 2, 1, 4, 3, 3, 1]^\top$ and $\bq^2=[5, 2, 3, 4, 4]^\top$.   One possible simultaneous decompression of $\bq^1$ and $\bq^2$ is
\begin{equation}\label{eqn:ExampleBV}
    {\small \bv^\top=
    \begin{matrix}
    [1, 1, 0, 1, 1, 1, 0, 1, 0, 1, 1, 0, 1, 0, 1, 1, 0, \\
    0, 1, 0, 1, 0, 0, 0, 1, 0, 0, 0, 1, 0, 0, 1, 1, 1, 0]
    \end{matrix}}
\end{equation} 
as $\bv$ $\delta_1$-compresses to $\bq^1$ and $\delta_2$-compresses to $\bq^2$.  In equation~\eqref{eqn:ExampleBV}, the top row shows indices $\{1,\dots,17\}$ and the bottom shows $\{18,\dots,35\}$.  Similarly, \begin{equation}
    {\small \bu^\top=
    \begin{matrix}
    [1, 1, 0, 1, 1, 1, 1, 1, 0, 0, 0, 0, 0, 1, 0, 1, 1, \\
    0, 0, 0, 1, 0, 1, 1, 1, 1, 0, 0, 1, 1, 1, 0, 0, 0, 0]
    \end{matrix}}
\end{equation} 
simultaneously compresses to $\bp^1=[3, 4, 3, 2, 2, 1, 3]^\top$ and $\bp^2=[6, 3, 2, 4, 3]^\top$.

Each theorem for single decompression trivially hold for simultaneous decompression.  The linearly independent set of simultaneous decompression constraints are
\begin{equation}\label{eqn:ConstraintMatrix}
\begin{bmatrix} 
\bOne^{\top}_{\ell} \\ 
\bOne^{\top}_{\frac{\ell}{\delta_1}}\otimes \begin{bmatrix}\bZer_{\delta_1-1} & \boldsymbol{I}_{\delta_1-1} \end{bmatrix} \\ 
\bOne^{\top}_{\frac{\ell}{\delta_2}}\otimes \begin{bmatrix}\bZer_{\delta_2-1} & \boldsymbol{I}_{\delta_2-1} \end{bmatrix}\\ \vdots \\ 
\bOne^{\top}_{\frac{\ell}{\delta_n}}\otimes \begin{bmatrix}\bZer_{\delta_n-1} & \boldsymbol{I}_{\delta_n-1} \end{bmatrix}\end{bmatrix}
\bv
=
\begin{bmatrix} \kappa \\ \bq^{1,\{-\}} \\ \bq^{2,\{-\}} \\ \vdots \\ \bq^{n,\{-\}}\end{bmatrix},
\end{equation}
where $\kappa=\sum_{i=0}^{\ell-1}v_i$ is the density constraint, $\bZer_{y}$ denotes a zero vector of length $y$, $\bq^{i,\{-\}}$ denotes all the $\delta_i$-compression excluding the $0$th index, and $\otimes$ is the {\it Kronecker} product.  
It follows from equation~\eqref{eqn:ConstraintMatrix} the number of linearly independent constraints defining simultaneous decompression is ${1+\sum_{i=1}^{n}(\delta_i-1)}$.  Similarly, if $\bq^i$ and $\bp^i$ are complementary sequences with PSD constant $\gamma$ for each ${i\in\{1,2,\dots,n\}}$, each vector generated through simultaneous decompression of all $\bq^i$ and $\bp^i$ respectively will satisfy $\sum_{j=1}^{n}\left(\delta_j-1\right)$ of the constraints defined by equation~\eqref{eqn:CompPSD}.  This is the case for pairs $\bq^1,\bp^1$ and $\bq^2,\bp^2$ with PSD constant $\gamma=18=(35+1)/2$. 

Notice from equation \eqref{eqn:ConstraintMatrix} the number of constraints set by simultaneous decompression increases as $\delta_i$ become imbalanced, with the most extreme example of this being $\delta_1=\ell$ and $\delta_2=1$.  The trade-off of this approach is the efficiency of generating complementary compressed sequences versus the efficiency of decompression.  Imbalanced factors increase complexity of generating complementary compressed sequences while relatively balanced factors, such as $\delta_1=7,\delta_2=5$, have greater complexity for generating decompressions.

The development of theory for simultaneous decompression will proceed similarly to single decompression.  The effects of circulant shifts and decimations will be addressed when considering simultaneous compressions, followed by an equation to determine the duplicity of decimation classes through simultaneous decompression.

Since $\ell=\prod_{j=1}^{n}\delta_j$ for $\delta_j$ pairwise relatively prime, the mapping 
\begin{equation}\label{eqn:psiIso}
    \psi:\Zl\mapsto \mathbb{Z}_{\delta_1}\times \mathbb{Z}_{\delta_2}\times\dots\times\mathbb{Z}_{\delta_n}=\bigtimes_{j=1}^{n}\mathbb{Z}_{\delta_j}
\end{equation}
with $\psi(1)=(1,1,\dots,1)$ constitutes a well known isomorphism.  Let $\bv$ be a vector of length $\ell$ with $\delta_i$-compression $\bq^i$ for $i\in\{1,2,\dots,n\}$.  Use of an isomorphism for simultaneous compression ensures each circulant shift on $\bv$ maps to a unique set of circulant shifts on $\{\bq^1,\bq^2,\dots,\bq^n\}$.  

Similarly let $\chi(j)=(j \ \text{mod} \ \delta_1, j \ \text{mod} \ \delta_2, \dots, j \ \text{mod} \ \delta_n)$ be the well known isomorphism,
\begin{equation}\label{eqn:chiIso}
    \chi:\Zlt\mapsto \mathbb{Z}^{\times}_{\delta_1}\times \mathbb{Z}^{\times}_{\delta_2}\times\dots\times\mathbb{Z}^{\times}_{\delta_n}=\bigtimes_{j=1}^{n}\mathbb{Z}^{\times}_{\delta_j}.
\end{equation}
It follows that decimation $d_j(\bv)$ yields a corresponding decimation on each $\delta_i$-compression, $d_j(\bq^i)$ for $i\in\{1,2,\dots,n\}$.  The following theorem determines the number of vectors within the decimation class of $\bv$ that simultaneously compress to $\{\bq^1,\bq^2,\dots,\bq^n\}$.

\begin{thm}\label{thm:MultiCount}
Let $\bv$ be a vector of length $\ell=\prod_{j=1}^{n}\delta_j$ such that ${\gcd(\delta_i,\delta_j)=1}$ for $i\neq j$, and $\bq^i$ its $\delta_i$-compression.  Let $G_{\bv}\leq\Zlt $ be the multiplier group of $\bv$, and $H_{i}\leq\mathbb{Z}_{\delta_i}^{\times}$ be the multiplier group of $\bq^i$.  The number of vectors within the decimation class of $\bv$ that simultaneously compress to $\{\bq^1,\bq^2,\dots,\bq^n\}$ is 
\begin{equation*}\label{eqn:multiCount}
    |D_{\bv,\{\bq^1,\bq^2,\dots,\bq^n\}}|=\frac{\prod_{i=1}^{n}|H_i|}{|G_{\bv}|}.
\end{equation*}
\end{thm}

\begin{proof}
Since no circulant shift of a vector with density relatively prime to $\ell$ may result in the same set of simultaneous compressions, it suffices to determine the duplicity of decimations.  By equation \eqref{eqn:chiIso}, for each $(j_1,j_2,\dots,j_n)$ there exists exactly one $j\in\Zlt$ such that $\chi(j)=(j_1,j_2,\dots,j_n)$ by equation~\eqref{eqn:chiIso}.   If there exists $i\in\{1,\dots,n\}$ such that $g \ (\text{mod} \ \delta_i)$ is not a multiplier of  $\bq^i$, then no $j\in\mathbb{Z}_{\delta_i}$ exists such that $d_g(\bq^i)=c_j(\bq^i)$.  Hence, $d_g(\bq^i)$ cannot decompress to $\bv$.  

Let $(h_1,h_2,\dots,h_n)\in H_1\times H_2\times\dots\times H_n$, and assume there exists some $g\in\Zlt $ such that $\chi(g)=(h_1,h_2,\dots,h_n)$, and $g$ is a multiplier of each $\bq^i$.  If $g$ is a multiplier of $\bv$, by definition $d_g(\bv)=c_j(\bv)$ for some $j\in\Zl$.  Then $d_g(\bv)$ simultaneously compresses to $\{\bq^1,\bq^2,\dots,\bq^n\}$ if and only if $j=0$.  

Now assume $g$ is not a multiplier of $\bv$, implying $c_j(d_g(\bv))\neq \bv$ for any $j\in\Zl$.  Since $\psi$ is an isomorphism (equation~\eqref{eqn:psiIso}) and $g$ is a multiplier of each $\bq^i$, there exists exactly one circulant shift $j\in\Zl$ such that ${c_j(d_g(\bq^i))=\bq^i}$ for each ${i\in\{1,2,\dots,n\}}$ and $c_j(d_g(\bv))$ compresses to $\{\bq^1,\bq^2,\dots,\bq^n\}$.  Thus, multiple vectors within $D_{\bv,\{\bq^1,\bq^2,\dots,\bq^n\}}$ result from simultaneous decompression if and only if $\chi(g)\in \left(H_1\times H_2\times\dots\times H_n\right)/\chi(G_v)$, implying $$|D_{\bv,\{\bq^1,\bq^2,\dots,\bq^n\}}|=\frac{|\bigtimes_{i=1}^{n}H_i|}{|G_{\bv}|}=\frac{\prod_{i=1}^{n}|H_i|}{|G_{\bv}|}.$$
\end{proof}

It follows from Theorem~\ref{thm:MultiCount} and Corollary~\ref{crly:DecClassCount} that $|D_{\bv,\{\bq^1,\bq^2,\dots,\bq^n\}}|$ divides $|D_{\bv,\bq^i}|\delta_i/\ell$ for any $i\in\{1,2,\dots,n\}$.  By Theorem~\ref{thm:DecClassRep}, the search space can be restricted to decimation class representatives for each $\delta_i$-compression, $\bq^i$.

The case $n=2$ is of particular interest as next five unsolved cases for LP existence are of length $\{77, 85, 87, 115, 117\}$, each the product of exactly two prime powers \cite{balonin2020three}.   Let 
\begin{equation*}\label{eqn:thetaIso}
    \theta:\{0,1\}^{\ell}\mapsto \{0,1\}^{\delta_1\times\delta_2}
\end{equation*}
be the mapping of binary vectors of length $\ell$ to arrays of size $\delta_1\times\delta_2$, with indices mapped as defined by equation \eqref{eqn:psiIso}.  When $n=2$, $\bOne_{\delta_1}^\top \bA=\bq^{2\top}$ and $\bA\bOne_{\delta_2} =\bq^{1}$.  For example, let $\bv$ be the vector given in equation~\eqref{eqn:ExampleBV}.  Then
\begin{equation*}\label{eqn:ExampleBVArray}
    \theta(\bv)=\bA=\begin{bmatrix}
1&	0&	1&	1&	1\\
1&	1&	0&	0&	0\\
0&	0&	0&	0&	1\\
1&	1&	0&	1&	1\\
0&	0&	1&	1&	1\\
1&	0&	1&	1&	0\\
1&	0&	0&	0&	0
    \end{bmatrix}.
\end{equation*}
As expected, $\bOne_{7}^\top\bA =\bq^{2\top}$ and $\bA\bOne_5=\bq^{1}$.  The problem of generating all such $\bA$ for fixed $\bq^1$ and $\bq^2$ is also known as enumerating all binary matrices with fixed marginals (BMFM)\cite{brualdi1980matrices}.  We focus on $n=2$ exclusively for the remainder of the paper.

\begin{thm}\label{thm:DFTConversion}
Let $\Omega_{\delta_1},\ \Omega_{\delta_2},$ and $\Omega_{\ell}$ be the DFT matrices of sizes $\delta_1$, $\delta_2$, and $\ell=\delta_1\delta_2$ respectively where $\gcd(\delta_1,\delta_2)=1$.  There exists unique ${z=\chi^{-1}\left(\delta_2^{-1},\delta_1^{-1}\right)}$ for $(\delta_2^{-1},\delta_1^{-1})\in\mathbb{Z}_{\delta_1}^{\times}\times\mathbb{Z}_{\delta_2}^{\times}$ such that
\begin{equation}
    \theta(\Omega_{\ell}(d_z(\bv)))=\theta(d_{z^{-1}}(\Omega_{\ell}\bv))=\Omega_{\delta_1}\theta(\bv) \Omega_{\delta_2}.
\end{equation}
\end{thm}

\begin{proof}
The uniqueness of $z$ is guaranteed by equation~\eqref{eqn:chiIso}.  It follows, $$\chi(z^{-1})=\left(\chi(z)\right)^{-1}=(\delta_2 \ \text{mod} \ \delta_1,\delta_1 \ \text{mod} \ \delta_2).$$  The relation $d_{z^{-1}}(\Omega_{\ell}\bv)=\Omega_{\ell}(d_{z}(\bv))$ for each $z\in\Zlt$ is well known \cite{Fletcher2001}.  To prove $\theta(d_{z^{-1}}(\Omega_{\ell}\bv))=\Omega_{\delta_1}\theta(\bv)\Omega_{\delta_2}$, let $\boldsymbol{M}=\Omega_{\delta_1}\theta(\bv) \Omega_{\delta_2}$, and $m_{r,c}$ be the element of $\boldsymbol{M}$ in row $r$, column $c$.

\begin{eqnarray*}
m_{r,c}=\sum_{j=0}^{\delta_2-1}e^{\left(cj\frac{2\pi i}{\delta_2}\right)} \sum_{k=0}^{\delta_1-1}e^{\left(rk\frac{2\pi i}{\delta_1}\right)} a_{k,j}\\
m_{r,c}=\sum_{j=0}^{\delta_2-1} \sum_{k=0}^{\delta_1-1}e^{\left(cj\frac{2\pi i}{\delta_2}\right)}e^{\left(rk\frac{2\pi i}{\delta_1}\right)} a_{k,j}\\
m_{r,c}=\sum_{j=0}^{\delta_2-1} \sum_{k=0}^{\delta_1-1}e^{\left(cj\frac{2\pi i}{\delta_2}\right)+\left(rk\frac{2\pi i}{\delta_1}\right)} a_{k,j}\\
m_{r,c}=\sum_{j=0}^{\delta_2-1} \sum_{k=0}^{\delta_1-1}e^{\left(cj\delta_1+rk\delta_2\right)\left(\frac{2\pi i}{\delta_1\delta_2}\right)} a_{k,j}\\
m_{r,c}=\sum_{j=0}^{\delta_2-1} \sum_{k=0}^{\delta_1-1}e^{\left(cj\delta_1+rk\delta_2\right)\left(\frac{2\pi i}{\ell}\right)} a_{k,j}
\end{eqnarray*}

It follows that 
$$m_{\left(x\delta_2^{-1},y\delta_1^{-1}\right)}=\sum_{j=0}^{\delta_2-1} \sum_{k=0}^{\delta_1-1}e^{\left(yj+xk\right)\left(\frac{2\pi i}{\ell}\right)} a_{k,j}=\theta(\mu_{\psi^{-1}(x,y)})$$
\end{proof}

Let pairs $(\bq^1,\bp^1)$ and $(\bq^2,\bp^2)$ are complementary sequences of sizes $\delta_1$ and $\delta_2$ respectively with PSD constant $\gamma$ and density $\kappa$.  Let $\bv$ and $\bu$ be solutions to BMFM with constraints $(\bq^1,\bq^2)$ and $(\bp^1,\bp^2)$ respectively.  It follows from equation~\eqref{eqn:CompConstPSD} that $PSD(\bv,g)+PSD(\bu,g)=\gamma$ for all $g\in\Zl/\{0,\Zlt\}$, and $PSD(\bv,j)+PSD(\bu,j)=\gamma$ need only be verified for $j\in\Zlt$ to determine if $\bv$ and $\bu$ are complementary.  These indices correspond to submatrix $\boldsymbol{M}^{\{-\}}$ of $\boldsymbol{M}$, where the $0$th row and column are removed from $\boldsymbol{M}$.  That is, if $\boldsymbol{M}=\Omega_{\delta_1}\theta(\bv)\Omega_{\delta_2}$ and $\boldsymbol{N}=\Omega_{\delta_1}\theta(\bu)\Omega_{\delta_2}$, then
$$\boldsymbol{M}^{\{-\}}\odot\overline{\boldsymbol{M}^{\{-\}}}+\boldsymbol{N}^{\{-\}}\odot\overline{\boldsymbol{N}^{\{-\}}}=\frac{\delta_1\delta_2+1}{2}\boldsymbol{J}_{(\delta_1-1)\times (\delta_2-1)}$$
where $\boldsymbol{J}_{m\times n}$ is the $m\times n$ matrix of ones, $\odot$ is the element-wise Hadamard product, and $\overline{\boldsymbol{M}^{\{-\}}}$ represents the element-wise complex conjugate of $\boldsymbol{M}^{\{-\}}$. 

Theorem~\ref{thm:Dursuns} further reduces the number of necessary PSD constraints from equation~\eqref{eqn:CompConstPSD} to $\left|\{\delta:\delta\in\Zl, \ \delta\mid \ell\}\right|$.

\begin{thm}\label{thm:Dursuns}
Let $\bv$ and $\bu$ be integer vectors indexed by $\Zl$ and $\delta_1\delta_2=\ell$.  If $\left(\PSD(\bv,\delta_2)+\PSD(\bu,\delta_2)\right)\in\mathbb{Q}$, then 
$$\left(\PSD(\bv,\delta_2)+\PSD(\bu,\delta_2)\right)=\left(\PSD(\bv,r\delta_2)+\PSD(\bu,r\delta_2)\right)$$
for each $r\in\mathbb{Z}_{\delta_1}^{\times}$.
\end{thm}

\begin{proof}
Let $\bp$ and $\bq$ be the $\delta_1$-modular compression of $\bv$ and $\bu$ respectively.  Then $\PSD(\bp,1)=\PSD(\bv,\delta_2)$ and $\PSD(\bq,1)=\PSD(\bu,\delta_2)$.  Let $\omega= e^{2\pi i/{\delta_1}}$. Then $\mathbb{Q}(\omega)$ is a field extension of $\mathbb{Q}$ of degree
$\phi(\delta_1)$, with 
${\rm Gal}(\mathbb{Q}(\omega)/\mathbb{Q}) \cong \mathbb{Z}_{\delta_1}^{\times}$.  Recall that ${\rm Gal}(\mathbb{Q}(\omega)/\mathbb{Q})$ is the group of 
all automorphisms of the field $\mathbb{Q}(\omega)$. Then for each $\sigma_r \in {\rm Gal}(\mathbb{Q}(\omega)/\mathbb{Q})$
 $$\sigma_r(\omega)=\omega^r$$ and $\sigma_r(z)=z$ for 
$z \in \mathbb{Q}$ and $r \in \mathbb{Z}_{\delta_1}^{\times}$.  By definition
$$
\PSD(\bp,1)=\sum_{i \in \mathbb{Z}_{{\delta_1}}}\omega^ip_i\overline{\sum_{i \in \mathbb{Z}_{{\delta_1}}}\omega^i p_i}.
$$
Then 
$$
\PSD(\bp,1)=\sum_{i \in \mathbb{Z}_{{\delta_1}}}\omega^ip_i\sum_{i \in \mathbb{Z}_{{\delta_1}}}\omega^{-i}p_i.
$$
Consequently, 
\small{\begin{equation}\label{eqn:sigma}
\PSD(\bp,1)=\sum_{i \in \mathbb{Z}_{{\delta_1}}}\omega^ip_i\sigma_{-1}\left(\sum_{i \in \mathbb{Z}_{{\delta_1}}}\omega^{i}p_i\right)
.
\end{equation}}
For $r \in \mathbb{Z}_{{\delta_1}}^{\times}$, applying $\sigma_r \in {\rm Gal}(\mathbb{Q}(\omega)/\mathbb{Q})$ to both sides of equation~(\ref{eqn:sigma}) yields
{\footnotesize
 \begin{equation}\label{eqn:sigmar}
\sigma_r\left(\PSD(\bp,1)\right)=\sigma_r\left(\sum_{i \in \mathbb{Z}_{{\delta_1}}}\omega^ip_i\right)\sigma_r\left(\sigma_{-1}\left(\sum_{i \in \mathbb{Z}_{{\delta_1}}}\omega^{i}p_i\right)\right).
\end{equation}}
Since ${\rm Gal}(\mathbb{Q}(\omega)/\mathbb{Q})$ is commutative equation~(\ref{eqn:sigmar})
 implies
  \begin{equation}\label{eqn:sigmar2}
\sigma_r\left(\PSD(\bp,1)\right)=\sigma_r\left(\sum_{i \in \mathbb{Z}_{{\delta_1}}}\omega^ip_i\right)\sigma_{-1}\left(\sigma_{r}\left(\sum_{i \in \mathbb{Z}_{{\delta_1}}}\omega^{i}p_i\right)\right).
\end{equation}
Notice 
\begin{equation*}
    \sigma_r\left(\PSD(\bp,1)\right)+\sigma_r\left(\PSD(\bq,1)\right)=\sigma_r\left(\PSD(\bp,1)+\PSD(\bq,1)\right),
\end{equation*}
and 
\begin{equation*}
    {\left(\PSD(\bp,1)+\PSD(\bq,1)\right)\in\mathbb{Q}}
\end{equation*}
implies $$\sigma_r\left(\PSD(\bp,1)+\PSD(\bq,1)\right)=\PSD(\bp,1)+\PSD(\bq,1).$$

Then by equation~\ref{eqn:sigmar2} we get
$$
\PSD(\bp,1)+\PSD(\bq,1)=\sum_{i \in \mathbb{Z}_{{\delta_1}}}\omega^{ri}p_i\overline{\sum_{i \in \mathbb{Z}_{{\delta_1}}}\omega^{ri}p_i}+\sum_{i \in \mathbb{Z}_{{\delta_1}}}\omega^{ri}q_i\overline{\sum_{i \in \mathbb{Z}_{{\delta_1}}}\omega^{ri}q_i}
$$
 for each $r \in \mathbb{Z}_{{\delta_1}}^{\times}$, and $$\PSD(\bv,r\delta_2)+\PSD(\bu,r\delta_2)=\PSD(\bv,\delta_2)+\PSD(\bu,\delta_2).$$
\end{proof}

Corollary~\ref{crly:Dursuns} follows directly from Theorem~\ref{thm:Dursuns} by the fact $(\ell+1)/2\in\mathbb{Q}$.
\begin{crly}\label{crly:Dursuns}
Vectors $\bu\in\{0,1\}^{\ell}$ and $\bv\in\{0,1\}^{\ell}$ constitute an LP if 
$$
\PSD(\bv,\delta)+\PSD(\bu,\delta)=\frac{\ell+1}{2} \ \forall \delta \mid \ell
$$
\end{crly}

It follows that if $(\bp^1,\bq^1)$ and $(\bp^2,\bq^2)$ are complementary sequences with simultaneous decompression $\bv$ and $\bu$ respectively, then $(\bv,\bu)$ constitute an LP if $$\PSD(\bv,1)+\PSD(\bu,1)=\gamma.$$


\section{A recursive method
\label{sec:FixedMargGenerationAlg}}

Several algorithms exist for enumerating the BMFM solution space.  Snijders' \cite{snijders1991enumeration} method recursively sets each matrix element according to a binary branching strategy based on the well known relation \begin{equation*}\label{eqn:bincoRecurse}
 \nchoosek{\ell}{k}=\nchoosek{\ell-1}{k}+\nchoosek{\ell-1}{k-1}.   
\end{equation*}
This equation is generalized to BMFM by Theorem~\ref{thm:RecurseCount}.

\begin{thm}
\label{thm:RecurseCount}
Let $\bq$ and $\bp$ be integer vectors of length $\delta_1$ and $\delta_2$ respectively.  Let $N(\bq,\bp)$ be the size of the solution set to BMFM defined by $\bp$ and $\bq$.  \begin{equation*}\label{eqn:recurseCount}
    N(\bq,\bp)=N(\bp,\bq)=\sum_{j=0}^{{{\delta_2} \choose {q_0}}-1}N(\bq_{1:\delta_1-1},\bp-\boldsymbol{r}_j)
\end{equation*}
where $\boldsymbol{r}_j$ is the binary vector of the $j$th ranked subset of size $q_0$ chosen from $\{0,1,\dots,\delta_2-1\}$.
\end{thm}

We call a BMFM infeasible if $N(\bq,\bp)=0$. It follows that a branch may be fathomed during matrix generation if the reduced problem is determined to be infeasible.  The next theorem provides the necessary and sufficient condition for determining if a BFMF is feasible (\cite{brualdi1980matrices}, pg.162).

\begin{thm}\label{thm:InfeasibleSet}
Let $\bq$ and $\bp$ be integer vectors of length $\delta_1$ and $\delta_2$ respectively.  $N(\bq,\bp)>0$ if and only if $t_{\bq,\bp}(I,J)\geq 0$ for all $I\subset\{1,\dots,\delta_1\}$ and ${J\subset\{1,\dots,\delta_2\}}$, where  \begin{equation*}
t_{\bq,\bp}(I,J)=|I||J|+\sum_{i\not\in I}q_i-\sum_{j\in J}p_j.
\end{equation*}
\end{thm} 

A necessary first step to employing a decompression based search is to locate all pairs of candidate compressed sequences that satisfy equation \eqref{eqn:CompConstPSD}.  By equation \eqref{eqn:chiIso}, the search for compressed complementary sequences can be initially reduced to searching across decimation class representatives.  It is well known that if $\bv$ and $\bu$ are complementary sequences, then $\bv$ and $d_{-1}(\bu)$ are also complementary \cite{Fletcher2001}.  This holds true for compressed sequences as well.  Rather, if $\bp^1$ and $\bq^1$ are compressed complementary sequences of size $\delta_1$, there exists  $S\subset\mathbb{Z}_{\delta_1}^{\times}$ such that each $s\in S$ is not a multiplier of $\bq^1$ but $PSD(d_s(\bq^1))=PSD(\bq^1)$.  Similarly,  there exists  $T\subset\mathbb{Z}_{\delta_2}^{\times}$ such that each $t\in T$ is not a multiplier of $\bq^2$ but $PSD(d_t(\bq^2))=PSD(\bq^2)$.  It follows that not all LP will arise as simultaneous decompressions from a single complementary pair of decimation class representatives.  

There are two options for handling this required duplicity of decimation class representation.  The first approach is to determine all $j\in\Zlt$ such that $(j \text{ mod }\delta_1)\in S$ and $(j \text{ mod }\delta_2)\in T$.  Each decompression of $\bp^1,\bp^2$ must be decimated by each $j$ and compared to the corresponding $\bq^1,\bq^2$ to determine if the decompressions, or decimations thereof, are complementary.  
The second approach is to conduct simultaneous decompression on each decimation of the set of complementary sequences, 
\begin{equation}\label{eqn:CompSets}
    \{(\bp^1,\bp^2),(d_s(\bq^1),d_t(\bq^2))\} \ \text{ for each } \ (s,t)\in (S\cup 1)\times (T\cup 1).
\end{equation}  Notice only one simultaneous decompression in each pair must be decimated.  We chose the latter approach as it has greater performance for conducting calculations in parallel.

We use the recursive method discussed by Brualdi~(\cite{brualdi1980matrices}, pg.195) based on Theorem \ref{thm:RecurseCount}  to enumerate the set of decompressions for fixed marginal vectors $\bq^1,\bq^2$.  This method recursively sets rows and columns of $\bA$, thereby reducing the BMFM's dimensions.  The method may be summarized as three stages.  The first stage determines if the BMFM is infeasible by Theorem \ref{thm:InfeasibleSet} and the search branch may be fathomed.  The second stage determines if there is a {\em trivial} assignment, such as a row that must filled with ones or a column that must be filled with zeros.  In this same stage, the method determines if the trivial assignments complete the matrix.  The third stage selects a row/column to branch on and generates all possible binary vectors of the corresponding size and density.  

We include DFT calculations at stages two and three of this method to improve computational efficiency of implementing the PSD test on completed decompressions.  PSD-based comparisons were used for our partial search over $\ell=77$.  Our enumerative LP search uses equation~\eqref{eqn:Complementary} for comparisons.  The latter stores two integer digits for each index and eliminates errors due to computational precision as $\sum_{j=0}^{\ell-1}v_jv_{j+g}$ is necessarily integer.  However, O$(\ell^2)$ calculations must be conducted for each decompression, reducing overall efficiency.

\section{Results}

\subsection{Enumeration for $\ell=55$}

An exhaustive search for LP of length $\ell=55$ was conducted. We identified $17$ inequivalent $5$-compression pairs and $2051$ inequivalent $11$-compression pairs satisfying equation \eqref{eqn:CompConstPSD}.  These were expanded to $31$ and $3,038$ pairs respectively per the decimations defined in equation~\eqref{eqn:CompSets}.  Each $5$ and $11$ pair combination yield $4$ BMFM for a total of $376,712$.  The total time required to generate the solution set of all BMFM is $793.3$hrs on a 16 cores 2.4Ghz computer running $128$ threads, or $101,542.4$ CPU hours.

A total of $18,042$ inequivalent LP were identified spanning $36,050$ decimation classes.  Known solutions were verified (up to decimation class equivalence) to be in the solution set \cite{balonin2020three,chiarandini2008heuristic,turner2019cocyclic}.

The {\em correlation energy} of solutions $(\bu,\bv)$, defined as:
\begin{equation*}
    \rho_{\bv}=\sum_{j=1}^{\frac{\ell-1}{2}}P_j(\bv)^2
\end{equation*}
have a range of $[5952, 8262]$ with a mode at $6942$.  The difference in observed consecutive correlation energy values is $2\ell$, the theoretical minimum change.  For comparison, the correlation energy of known $\ell=55$ LP are (6722)\cite{balonin2020three}, (7382)\cite{chiarandini2008heuristic}, and (7052, 7162, 7602)\cite{turner2019cocyclic}.  The tightest known lower bound on minimum correlation energy is $5292$, and is known to not be tight for $\ell=55$ as there exists no $(55,28,14)$ difference set.  Rather, no binary vector of size $\ell=55$ and density $\kappa=28$ is self-complementary \cite{arasu1991mann}. 

\begin{figure}
    \centering
    \includegraphics[scale=0.60]{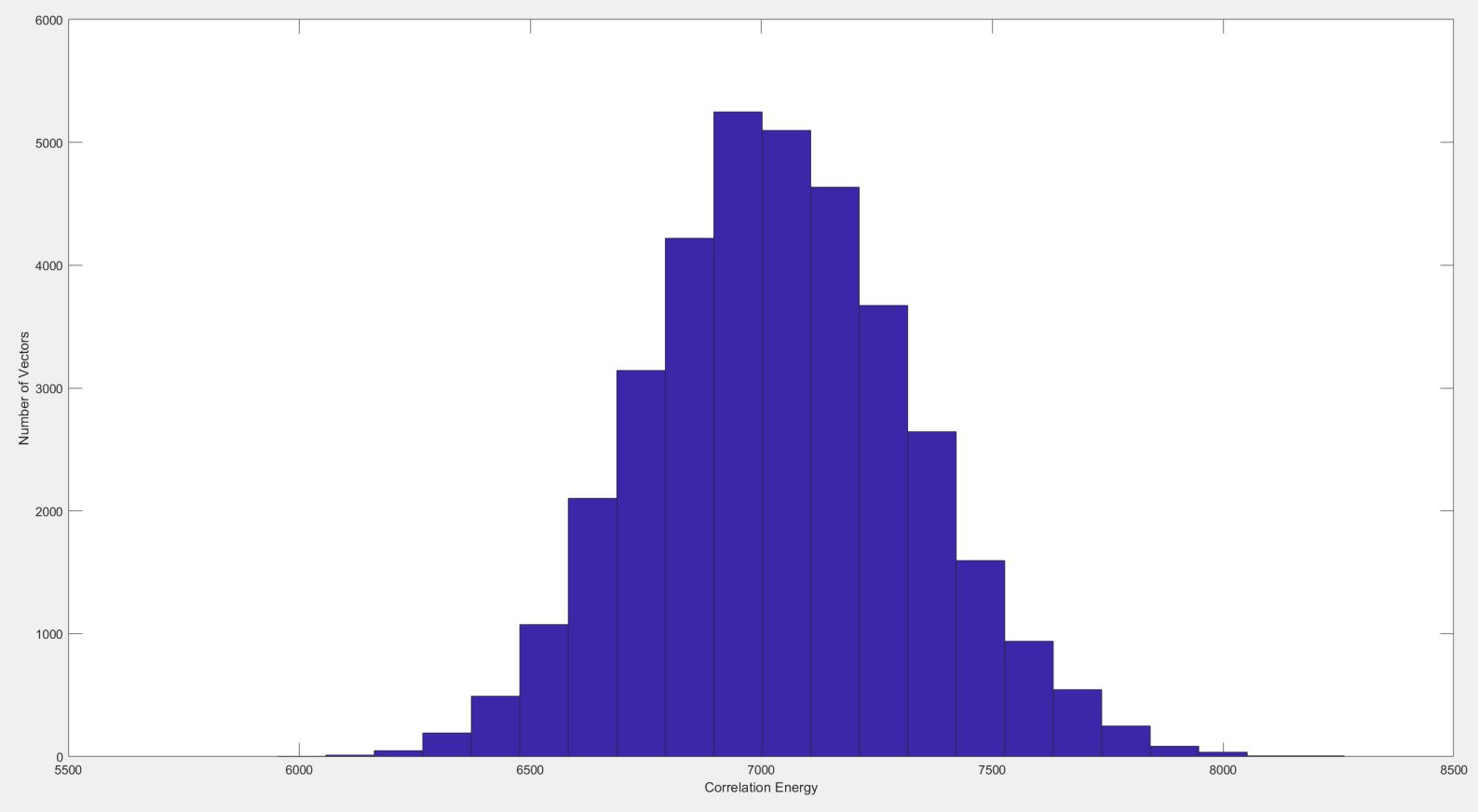}
    \caption{Histogram of correlation energy for $\ell=55$ vectors belonging to an LP}
    \label{fig:CorrEng}
\end{figure}

\subsection{Results for $\ell=77$}

A partial search for LP of length $77$ with $\delta_1=7$ and $\delta_2=11$ produced the LP in Table \ref{tab:L77Sols}.  This search was conducted with 168 threads on an Intel I5-3570 CPU, 4 core, 3.40 GHz processor, and required approximately 182,280 CPU hours.

\begin{table}[ht]
    \centering
{\footnotesize{
\begin{eqnarray*}
    \bu=\begin{matrix}
    [0,1,0,1,0,1,1,0,1,0,0,1,1,1,1,1,0,1,0,1,1,0,0,1,1,0,0,1,0,1,0,0,0,1,1,0,0,0,0, \\
    0,1,1,1,1,1,0,1,1,0,1,0,1,0,0,1,0,0,1,0,1,0,1,1,1,1,1,0,1,0,0,0,1,0,0,0,0,1]^\top
    \end{matrix}\\
    \bv=\begin{matrix}
    [0,1,1,0,0,1,1,0,0,1,1,0,1,1,1,0,0,0,0,1,1,1,1,1,0,1,1,1,1,1,0,0,1,0,1,1,0,1,1, \\
    0,0,0,0,0,0,1,1,1,0,0,0,1,0,0,0,0,1,1,0,1,0,0,0,1,0,0,0,1,1,0,0,1,1,1,1,0,1]^\top
    \end{matrix}
\end{eqnarray*}
}}
    \caption{$\ell=77$ LP (Indices: $(1,\dots,39)$ top row, $(40,\dots,77)$ bottom row)}
    \label{tab:L77Sols}
\end{table}

\section*{Acknowledgement}
This research was sponsored by AFIT Graduate School of Engineering and Management's Faculty Research Council, grant 2019-213.  The views expressed in this article are those of the authors and do not reflect the official policy or position of the United States Air Force, Department of
 Defense, or the U.S. Government



\newpage
\bibliographystyle{plain}
\bibliography{BibTex.bib}







\end{document}